\documentclass[12pt]{article}

\usepackage{amssymb}
\usepackage{amsmath}
\usepackage{amsthm}
\usepackage[alphabetic, initials]{amsrefs}
\allowdisplaybreaks

\newtheorem{theorem}{Theorem}[section]
\newtheorem{lemma}[theorem]{Lemma}

\newtheorem{proposition}[theorem]{Proposition}
\newtheorem{claim}[theorem]{Claim}
\newtheorem{corollary}[theorem]{Corollary}

\newtheorem{conjecture}[theorem]{Conjecture}

\newtheorem*{theoremA}{Theorem \ref{thm:discrete-propagation-of-smallness}'}

\theoremstyle{definition}   
\newtheorem{definition}[theorem]{Definition}
\newtheorem*{example}{Example}

\theoremstyle{remark}    
\newtheorem{remark}[theorem]{Remark}
\newtheorem*{remark*}{Remark}
\newtheorem*{notation*}{Notation}

\newcommand{\ep}{\varepsilon}
\newcommand{\eps}{\varepsilon}

\newcommand{\Eb}{\mathbb{E}}

\newcommand{\Nb}{\mathbb{N}}
\newcommand{\R}{\mathbb{R}}
\newcommand{\Rb}{\mathbb{R}}
\newcommand{\Zb}{\mathbb{Z}}

\newcommand{\Z}{\mathbb{Z}}

\newcommand{\Cay}{\mathrm{Cay}}
\newcommand{\Prob}{\mathrm{Prob}}

\begin{document}
\title{Harmonic functions on the lattice: Absolute monotonicity and
propagation of smallness}
\author{Gabor Lippner and Dan Mangoubi}
\date{}
\maketitle
\begin{abstract}
In this work we establish a connection between two classical 
notions, unrelated so far: Harmonic functions on the
one hand and absolutely monotonic functions on 
the other hand. We use this to prove convexity type and propagation
of smallness results for harmonic functions on the lattice.
\end{abstract}

\section{Introduction}
The aim of the present paper is to discuss
convexity properties of discrete harmonic functions.
One of our main discoveries is that if $u$ is a harmonic function then $\Delta^k u^2$ is non-negative for every $k\in\Nb$, and that this fact can be used 
to prove convexity type results.
A minor byproduct of this work is an elementary new proof of
the well known Liouville property and finite dimensionality of harmonic functions of bounded polynomial growth in $\Zb^d$.
Some of our results can be adapted also to harmonic functions in $\Rb^d$,
but we do not pursue this direction here, since
working in the discrete world, we are faced with
 several challenges which do not exist in the continuous world.
\subsection{Background: Hadamard's Three Circles Theorem, Agmon's Theorem.}
The  connection between  holo\-morphic functions and convexity
goes back to Hadamard. For $f$ holomorphic, let
$M(r):=\max_{B(r)} |f|$. Hadamard proved the Three Circles Theorem: $M(2r)\leq \sqrt{M(r)M(4r)}$ or equivalently,
$\log M(r)$ is a convex function as a function of $\log r$.
Once this theorem is known, a version for harmonic function $u$
is naturally sought after. It is readily seen that $M_u(r):=\max_{B(r)}|u|$ is a convex function as a function of $\log r$. However, only
approximately is it true that $\log M_u$ is a convex function of $\log r$. The approximate logarithmic convexity of $M_u(r)$ is due to
Landis' school \cite{landis63}*{Ch. II.2}.
Apparently, Agmon was the first to  observe that
if one replaces the function $M_u(r)$ with an $L^2$-version on a sphere
then one gets a precise logarithmic convexity result:
\begin{theorem}[\cite{agmon66}]
\label{thm:agmon}
Let $u$ be a harmonic function defined in the open ball of radius $R$, $B(0, R)\subset\Rb^d$.
Let the $L^2$-growth function be defined by
\begin{equation}
\label{def:q-agmon}
q_u(r):=\frac{1}{|S^{d-1}(r)|}\int_{S^{d-1}(r)} u^2 d\sigma_r.
\end{equation}
Then $\log q_u$ is a convex function of $\log r$ for $\log r<\log R$.
\end{theorem}
Here we should also mention that the idea to consider 
the integral of $u^2$ on arcs (in two dimensions) 
is due to Carleman, and in particular, a second order differential inequality closely related to the logarithmic convexity was proved 
in~\cite{carleman-1933}. The observation of Agmon was
that in a specific setting Carleman's differential inequality can be
strengthened to obtain the logarithmic convexity.
\subsection{Main result I: Discrete absolute monotonicity}
In fact, it turns out that a stronger property of $q_u(r)$ holds. We recall the following definition
\begin{definition}[\cite{bernstein1914}]
A $C^{\infty}$-function $f:[0,T)\to\Rb$ is called
\emph{absolutely monotonic} in $[0, T)$ if $f$ and all its derivatives are non-negative.
\end{definition}

A fundamental property of absolutely monotonic functions is given by
a theorem of S.~Bernstein (see also~\cite{widder41}*{Ch. IV}):
\begin{theorem}[\cite{bernstein1914}]
\label{thm:bernstein}
An absolutely monotonic function in $[0, T)$ is real-analytic in $(0, T)$. Moreover, it extends to a real-analytic function in $(-T, T)$.
\end{theorem}

We observe
\begin{theorem}
\label{thm:abs-mono}
If $u:B(0, R)\to\Rb$ is harmonic, then the function $q_u(r)$ is absolutely monotonic in $[0, R)$.
\end{theorem}

Using Theorem~\ref{thm:bernstein}, 
it is an exercise to check that absolute monotonicity
implies logarithmic convexity on the scale of $\log r$~\cite{polya-szego-vol-I}*{part II, problem~123}.
Theorem~\ref{thm:abs-mono} is easily understood by considering the spherical harmonics expansion of harmonic functions.

A central idea we try to convey in this paper is that
the strengthened
Agmon's Theorem~\ref{thm:abs-mono} lends itself more naturally
to a  discrete analogue than Theorem~\ref{thm:agmon}.
To explain this, we let $(X_n)_{n=0}^{\infty}$ be a simple random walk on $\Zb^d$
starting at~$0$, and $(Y_t)_{t\geq 0}$ be a continuous time random walk on $\Zb^d$ starting at~$0$.
We introduce the following discrete $L^2$-growth functions

\begin{definition}
Let $B_R\subseteq \Zb^d$ be the closed ball of radius $R$ centered at $0$.
The discrete growth function of $u: B_R \to \Rb$ is defined by 
$$\forall 0\leq n\leq R, \quad Q_u(n):=\Eb \left(u(X_n)^2\right).$$
If $u$ is globally defined we set
$$Q_{c, u}(t):=\Eb\left(u(Y_t)^2\right).$$
\end{definition}

\begin{remark}\label{rem:infiniteQ}
$Q_{c, u}(t)$ could be infinite. However, there exists $T\geq 0$
such that $Q_{c, u}(t)<\infty$ for $0<t<T$ and $Q_{c, u}(t)=\infty$
for $t>T$. Indeed, if we denote by $p_c(t, x)$
the continuous time heat kernel on $\Zb^d$, then it follows directly from the heat equation and the fact that $p_c \geq 0$  that $e^{t}p_c(t, x)$ is non-decreasing in $t$. The claim follows since
$Q_{c, u}(t, x)=e^{-t}\sum_{x\in \Zb^d} u(x)^2 e^tp_c(t, x)$. 
\end{remark}

\begin{remark*} Observe that unlike the $\Rb^d$ version~\eqref{def:q-agmon}, where we consider
spheres in space, here we consider spheres in time.
\end{remark*}

We define
\begin{definition}
\label{def:abs-mono-discrete-derivative}
Let $f:\Nb_0\cap[0, R]\to\Rb$ be a discrete function.
We say that $f$ is \emph{absolutely monotonic in $[0, R]$ in the discrete sense}
if $f^{(k)}(n)\geq 0$ for all $k, n\in\Nb_0$ such that $k+n\leq R$.
Here $f^{(k)}$ is the $k$-th forward discrete derivative, i.e.
$f'(n)=f(n+1)-f(n)$ and $f^{(k)}:={f^{(k-1)}}'$.
\end{definition}

\begin{example} $f(n)=\binom{n}{k}$ is absolutely monotonic in $\Nb_0$ in the discrete sense.
\end{example}

A main result we prove is
\begin{theorem}
\label{thm:discrete-abs-mono}
Let $u: B_R\to \Rb$ be harmonic. Then, 
the function $Q_u$ is absolutely monotonic in $[0, R]$ in the discrete sense.
If $u$ is globally defined and $Q_{c,u}$ is finite on $[0,T)$, then it is absolutely monotonic in $[0,T)$.
\end{theorem} 

Besides the main applications of Theorem~\ref{thm:discrete-abs-mono} discussed below, we also immediately obtain a new simple proof of the finite dimensionality of the space of harmonic functions of bounded polynomial growth.

\begin{corollary}\label{cor:zeros-vs-growth}
Let $u:\Zb^d\to\Rb$ be a non-zero harmonic function. Suppose it has a degree $M$-polynomial growth,
i.e. $\limsup_{|x|\to\infty}|u(x)|/|x|^M< \infty$.
Then~$u$ must be a polynomial of degree at most $M$.
In particular, the space of degree $M$-polynomial growth harmonic functions is finite dimensional.
Moreover, the function $u$ cannot vanish identically on a ball of radius $M$. 
\end{corollary}

\subsection{Main result II: Discrete Three Circles Theorems}
\label{sec:three-circles}
Our main application of Theorem~\ref{thm:discrete-abs-mono} gives
two discrete analogues of Hadamrad's Three Circles Theorem.

It is easier to begin with the setting of a globally defined
harmonic function and the continuous time random walk.
\begin{theorem}
\label{cor:continuous-propagation-of-smallness}
Let $u:\Zb^d\to \Rb$ be harmonic.  Then
$$Q_{c, u}(2t)\leq \sqrt{Q_{c, u}(t)Q_{c, u}(4t)}.$$
\end{theorem}

A  version of Theorem~\ref{cor:continuous-propagation-of-smallness} for the discrete time random walk turns out to be more difficult.
 The reason is that an error term must appear (for $d\geq 2$).
Indeed, for any $R\in\Nb$ there exists a harmonic polynomial $u:\Zb^d\to\Rb$ such that $u|_{B_{2R}}$ is not zero and $u|_{B_R}\equiv 0$.
 Then, $Q_u$ cannot satisfy an inequality of the form
 $Q_u(2n)\leq\sqrt{Q_u(n)Q_u(4n)}$ for all $n\in\Nb$.

\begin{theorem}
\label{thm:discrete-propagation-of-smallness}
Let $u: B_{4 R}\to \Rb$ be harmonic and let \mbox{$0 \leq \eps \leq 0.5$}. Then  
\begin{equation}
\label{ineq:propagation-with-error} 
Q_u(2n)\leq \sqrt{e^{n^{-2\eps}}Q_u(n)Q_u(4 n)} + 
 2^{-n^{0.5-\eps}}Q_u(4n ),
\end{equation}
for all $16 < n\leq R$.
\end{theorem}


 Observe that the constant $e^{n^{-2\eps}}$ tends to~$1$ and the error term $2^{-n^{0.5-\eps}}$ goes to $0$ as $n$ goes to infinity if $0<\eps<0.5$. Theorem~\ref{thm:discrete-propagation-of-smallness} is inspired from~\cite{guadie-malinnikova-unique-continuation}, 
 where an $L^\infty$ variant
with essentially the same error term is proved.

  In case we have an a priori bound on the growth of $u$
  we show that the error term can be dropped for large $n$: 
 \begin{theorem}
 \label{thm:propagation-with-error-omitted}
 Let $u:\Zb^d\to\Rb$ be a harmonic polynomial of degree $M$. Let $0\leq \eps<0.5$. Then
 \begin{equation}
 \quad Q_u(2n)\leq \sqrt{e^{n^{-2\eps}}Q_u(n)Q_u(4n)}
 \end{equation}
  for all $n$ such that $n^{1-{2\eps}}>M^2$ and $n > 16$.
 \end{theorem}
 
 It would be most interesting to understand the sharp error term
 in Theorem~\ref{thm:discrete-propagation-of-smallness}.
 If we do not
 fix the dimension of the lattice, the  error term 
 proved is optimal.

 \begin{theorem}\label{thm:weak-conjecture}
 For all $\eps>0,$  $n_0\in N$, $C>0$ and $d>d_0(\eps, n_0, C)$ 
 there exists a harmonic function $u:\Zb^d \to \R$
such that
\begin{equation}\label{eq:error-term} Q_u(2n)> C \sqrt{Q_u(n)Q_u(4n)} + 2^{-n^{0.5+\eps}}Q_u(4n)
\end{equation}
for some $n>n_0$.
 \end{theorem}

However, we conjecture that the error term proved is
optimal in every dimension.

 \begin{conjecture}
\label{conj:sharp-error}
 For all $\eps>0$, $n_0\in N$, $C>0$ and $d\geq 2$,
there exists a harmonic function $u : \Z^d \to \R$ such that
$$Q_u(2n)> C \sqrt{Q_u(n)Q_u(4n)} + 2^{-n^{0.5+\eps}}Q_u(4n)$$
for some $n>n_0$.
\end{conjecture}
\noindent In Section~\ref{sec:conj} we discuss Conjecture~\ref{conj:sharp-error}.

\begin{remark}
It is interesting to point out the connection between the constant in front of the main term in the RHS of~\eqref{ineq:propagation-with-error} and the exponent $-n^{\beta}$ in the error term in \eqref{ineq:propagation-with-error}. We prove that if $\beta<0.5$ then the constant tends to~$1$ for large $n$, and we believe that $\beta$ cannot be made bigger than $0.5$ even at the expense of an arbitrarily large constant. We do not know what to expect when $\beta=0.5$, besides what we prove in Theorem~\ref{thm:discrete-propagation-of-smallness}. 
\end{remark}
   
  It is worth noting that if the ratios $1:2:4$ in Theorem~\ref{thm:discrete-propagation-of-smallness} are perturbed, the error term drops dramatically, hinting at a delicate phase change phenomenon, which further motivates our interest in the optimal error term.
\begin{theorem}
 \label{thm:discrete-propagation-of-smallness-1-2-5}
 Let $u:B_{5R}\to\Rb$ be harmonic.
 Let $0<\delta<1/4$. Then
 $$Q_u(2n)\leq \sqrt{Q_u(n)Q_u\left(\lceil(4(1+\delta)n\rceil\right)} + 2^{-2 n\delta}Q_u\left(\lceil(4(1+\delta)n\rceil\right)$$
 for all $0\leq n\leq R$.
 \end{theorem}
 
However, for three concentric circles of any aspect ratio,
if one adjusts the main term correspondingly the error term 
is stable.
 \begin{theorem}
 \label{thm:discrete-propagation-of-smallness-aspect-ratio}
 Let $1<P<pP$.  Then,  for any harmonic $u: B_{pP}\to \Rb$ and any $0\leq \eps \leq 0.5$
 \begin{equation*}
 \quad Q_u(\lfloor P n\rfloor)\leq e^{c n^{-2\ep}} Q_u(n)^{\alpha}Q_u(\lceil pP n \rceil)^{1-\alpha} + 
  p^{-n^{0.5-\eps}}Q_u(\lceil pP  n\rceil),
 \end{equation*}
 for all $n_0\leq n\leq R$, where $n_0=n_0(p, P)$, $P^{\alpha}=p^{1-\alpha}$, $c = 2(\alpha P+(1-\alpha)\frac{1}{p} - 1)$.
 \end{theorem}
The proof of Theorem~\ref{thm:discrete-propagation-of-smallness-aspect-ratio}
is given for $\alpha=1/2$ (see Section~\ref{sec:three-circles}, Theorem~\ref{thm:discrete-propagation-of-smallness}').
The other cases are obtained by a slight modification and we omit the details.

\paragraph{Organization of the paper.}  In Section~\ref{sec:high-powers} we explain our crucial observation: Positivity of $\Delta^k u^2$ for a harmonic function $u$, and then we use it to prove the discrete analogue of the strengthened Agmon's theorem, Theorem~\ref{thm:discrete-abs-mono}, and Corollary~\ref{cor:zeros-vs-growth}. In Section~\ref{sec:logconv} we prove the discrete versions of the Three Circles Theorem announced in Section~\ref{sec:three-circles}. In Section~\ref{sec:weak-conj} we show that the error term in 
Theorem~\ref{thm:discrete-propagation-of-smallness} is optimal 
in a weak sense (Theorem~\ref{thm:weak-conjecture}). In Section~\ref{sec:conj} we  discuss Conjecture~\ref{conj:sharp-error}.

\paragraph{Acknowledgments.}
Our preliminary results in the discrete setting involved the growth function at stopping times on spheres. We heartily thank Gady Kozma for his suggestion to  try replacing the spheres in space by spheres in time, as  this idea simplified and clarified most of our results.
We are very grateful to Eugenia Malinnikova who explained to one of us her joint work~\cite{guadie-malinnikova-unique-continuation} which in turn inspired our statements of Theorem~\ref{thm:discrete-propagation-of-smallness} and Conjecture~\ref{conj:sharp-error}. We also thank Eugenia for her valuable comments and helpful discussions.
We owe our gratitude to  Matthias Keller for helping
us explore relevant directions.
We thank Itai Benjamini, Bo'az Klartag and Gady Kozma for helpful discussions.
We are very grateful to Shing-Tung Yau for his continuous support
of both of us. 
We thank the anonymous referees for their valuable comments and for bringing to our attention the papers~\cite{jls-duke14}
and~\cite{carleman-1933}.
 G.L. gratefully acknowledges the support of AFOSR grant FA9550-09-1-0090.
D.M. gratefully acknowledges the support of ISF grant no. 225/10.
Both authors gratefully acknowledge the support of BSF grant no. 2010214.

\section{Laplacian powers of a harmonic function squared}
\label{sec:high-powers}

In this section we prove Theorem~\ref{thm:discrete-abs-mono} and
conclude Corollary~\ref{cor:zeros-vs-growth}.
 The heart of the matter in the proof of Theorem~\ref{thm:discrete-abs-mono}
 is the following
observation which we believe to be interesting in its own right.
\begin{theorem}
\label{thm:high-powers-laplacian}
Let $u$ be a harmonic function on a Cayley graph of a finitely
generated Abelian group or on $\Rb^d$. Then 
 $\Delta^k (u^2)$ is non-negative for all $k\geq 0$.
\end{theorem}

We emphasize that in the proof of Theorem~\ref{thm:discrete-abs-mono}, the harmonicity of $u$ is used only through the application of 
Theorem~\ref{thm:high-powers-laplacian}.
\begin{proof}[Proof of Theorem~\ref{thm:high-powers-laplacian}]
The theorem follows by induction on $k$ from Claim~\ref{claim:key} below.
\end{proof}

\begin{claim}\label{claim:key} Let $(A, S)$ be a finitely generated
Abelian group with a finite set of generators $S$.
 Let $u:\Cay(A, S)\to\Rb$ or
$u:\Rb^d\to\Rb$ be harmonic, then 
$\Delta  (u^2)$  is a sum of squares of harmonic functions.

\end{claim}
\begin{proof}[Proof of Claim~\ref{claim:key}]
We give the proof for Cayley graphs. The proof in $\Rb^d$ 
is similar.
\begin{align*} 
(\Delta u^2)(x)  &= \frac{1}{|S|}\sum_{s \in S} \left(u(x + s)^2 - u(x)^2\right) \\ 
&=\frac{1}{|S|}\sum_{s \in S} (u(x + s) - u(x))^2 
+ 2u(x)\left(u(x + s) - u(x)\right) \\ 
&= \frac{1}{|S|}\sum_{s\in S} (u(x+s)-u(x))^2 + 2u(x) (\Delta u)(x) \\ &= \frac{1}{|S|}\sum_{s\in S} (u(x+s)-u(x))^2.
\end{align*}
Finally, since translation commutes with the Laplacian
on Abelian groups, $x\mapsto u(x+s)-u(x)$ is a harmonic function.
\end{proof}

\begin{remark}
\label{rem:derivatives}
For  later reference we record here a convenient related formula.
For a generator $s\in S$ we set
\begin{equation*}
u_s(x):=u(x+s)-u(x)
\end{equation*}
Iterating the proof of Claim~\ref{claim:key} shows
that for a harmonic function $u$
\[ \Delta^k u^2 = \frac{1}{|S|^k} \sum_{s_1,s_2,\dots,s_k \in S} (u_{s_1s_2\dots s_k})^2.\]
\end{remark}

\subsection{Proof of Theorem~\ref{thm:discrete-abs-mono}}
We prove absolute monotonicity of the growth function of a harmonic
function.
We begin with the following nice identity:
\begin{lemma}
\label{lem:identity}
Let $f:\Zb^d\to [0,\infty)$ be any function.
Let $E_f(n):=\Eb f(X_n)$. If $\Eb |f(Y_t)|<\infty$ let 
$E_{c, f}(t):=\Eb f(Y_t)$.
Then
$$E_f^{(k)}(n) = \Eb ((\Delta^k f) (X_n)),$$ and
 $$E_{c, f}^{(k)}(t) = \Eb ((\Delta^k f)(Y_t)).$$
\end{lemma}
\begin{proof}[Proof of Lemma~\ref{lem:identity}]
The Lemma follows from the case $k=1$.
Let 
$$p(n, x)=\Prob(X_n=x).$$
The function $p$ satisfies the heat equation:
\begin{equation*}
p(n+1, x)-p(n, x) = (\Delta p)(n, x)\ ,
\end{equation*}
where $\Delta p$ is the Laplacian with respect to the space parameter,~$x$.
Hence, 
\begin{multline*}
E_f'(n)=E_f(n+1)-E_f(n)=\sum_x f(x)(p(n+1, x)-p(n, x))
\\=\sum_x f(x) (\Delta p)(n, x) = \sum_x (\Delta f)(x) p(n, x)
=\Eb ((\Delta f)(X_n)),
\end{multline*}
where we have used the self adjointness of the Laplace operator.

In the continuous case the line of proof is the same, but we have to treat convergence issues. First, we check that differentiation
of the infinite sum can be done term by term.
To that end, it is sufficient to check that there exists a $\delta > 0$ such that for all $t \in [0,T)$ 
$$\sum_x \sup_{|t'-t|<\delta}
|\partial_t p_c(t', x)| |f(x)|<\infty\ .$$
 We have
  $$|\partial_t p_c(t',x)| = |\Delta p_c(t',x)| \leq p_c(t',x) + (1/4)\sum_{y\sim x} p_c(t',y)\ .$$
  Observe that $p_c(t', v)\leq e^{\delta} p_c(t, v)$  for any $v\in\Zb^d$   by Remark~\ref{rem:infiniteQ}. Hence, it is enough to show that
  $\sum_x \sum_{y\sim x}p_c(t, y) |f(x)|<\infty$.
However, for any $t$ and $y\sim x$ we have $p_c(t+\delta,x) = \sum_z p_c(t,z)p_c(\delta ,x-z) \geq p_c(\delta,e_1)p_c(t,y)$. Then,
\[ \sum_x \sum_{y\sim x} p_c(t,y)|f(x)| \leq \frac{4}{p_c(\delta,e_1)} \sum_x p_c(t+\delta,x)|f(x)| \leq
 \frac{4e^{\delta}}{p_c(\delta, e_1)}E_{c,|f|}(t)<\infty\]
 as desired.
 Second, we need to justify the reordering of the terms that is used to prove the identity $\sum_x f(x) (\Delta p_c)(t,x) = \sum_x (\Delta f)(x) p_c(t, x) $. For this, it is enough to check that
 $$\forall 1\leq j\leq 2d,\quad \sum_x p_c(t, x+e_j) |f(x)|<\infty\ . $$
 Again, we apply the same argument as before:
 $$p_c(t, x+e_j)\leq\frac{e^\delta p_c(t, x)}{p_c(\delta, e_1)}\ .$$
\end{proof}

We have all the ingredients now to present
\begin{proof}[Proof of Theorem~\ref{thm:discrete-abs-mono}]
By Lemma~\ref{lem:identity}
\begin{equation*}
Q_u^{(k)}(n)=\Eb \left((\Delta^k u^2)(X_n)\right)\ .
\end{equation*}
By Theorem~\ref{thm:high-powers-laplacian} the expression
on the RHS is non-negative.
The proof for $Q_{c, u}$ is similar.
\end{proof}

\subsection{The Newton series of the growth functions}
In this section we prove the following useful formulas:
\begin{theorem}
\label{thm:finite-computation}
Let $u:B_R\to\Rb$ be an arbitrary function. Then
\[\forall 0\leq n\leq R,\quad Q_u(n)=\sum_{k=0}^R (\Delta^k u^2)(0)\binom{n}{k}.\] 
In addition,  
if $u$ is globally defined and $Q_{c, u}(t)<\infty$
then
\[ Q_{c, u}(t)=\sum_{k=0}^{\infty} (\Delta^k u^2)(0)\frac{t^k}{k!}.\]
\end{theorem}

The proof of Theorem~\ref{thm:finite-computation} is based on Lemma~\ref{lem:identity}, Theorem~\ref{thm:bernstein} and the following classical theorem
on finite differences.

Let $F:\Nb_0\to\Rb$ be a discrete function.
Let $F^{(k)}$ be as in Definition~\ref{def:abs-mono-discrete-derivative}.

\begin{theorem}[Newton series]
\label{thm:discrete-taylor}
The function $F:\Nb_0\to\Rb$ can be uniquely written in the form 
\begin{equation*}
 F(n)  = \sum_{k=0}^\infty a_k \binom{n}{k}.
\end{equation*}
Moreover, $a_k = F^{(k)}(0)$.
\end{theorem}

Although standard, we reproduce the short proof for completeness.
\begin{proof}[Proof of Theorem~\ref{thm:discrete-taylor}]
We can write
$$F^{(k)}(n)=\sum_{j=0}^{k} (-1)^{k+j}\binom{k}{j}F(n+j).$$
Hence
\begin{multline*}
\sum_{k=0}^{\infty} F^{(k)}(0)\binom{n}{k}
=\sum_{k=0}^{n}\sum_{j=0}^{k} (-1)^{k+j}\binom{k}{j}
\binom{n}{k}F(j)\\
=\sum_{j=0}^{n}(-1)^jF(j)\sum_{k=j}^{n} (-1)^{k}\binom{k}{j}
\binom{n}{k}=
\\=\sum_{j=0}^{n} (-1)^j F(j)\sum_{l=0}^{n-j}
(-1)^{l+j}\binom{l+j}{j}\binom{n}{l+j}
\\=\sum_{j=0}^{n} F(j)\binom{n}{j}\sum_{l=0}^{n-j} (-1)^l
\binom{n-j}{l}=\sum_{j=0}^{n}F(j)\binom{n}{j}\delta_{n-j, 0}
=F(n).
\end{multline*}
\end{proof}

\begin{proof}[Proof of Theorem~\ref{thm:finite-computation}]
Let $f:\Zb^d\to\Rb$ be any function.
By Lemma~\ref{lem:identity} and Theorem~\ref{thm:discrete-taylor}
we get
$$E_f(n)=\sum_{k=0}^n E_{\Delta^k f}(0)\binom{n}{k}.$$
It only remains to observe that
$E_{\Delta^k f}(0)=(\Delta^k f)(0)$.
If we take $f=u^2$ in~$B_R$ and $f=0$ outside $B_R$ we get the first part of the theorem.

We move to the second part.
By Theorem~\ref{thm:discrete-abs-mono} and  Theorem~\ref{thm:bernstein} we know that
\begin{equation}
\label{eqn:Q_cu-taylor}
Q_{c, u}(t)=\sum_{k=0}^{\infty} Q_{c, u}^{(k)}(0)\frac{t^k}{k!}\ .
\end{equation}
The second part of the theorem now follows from
formula~\eqref{eqn:Q_cu-taylor} and Lemma~\ref{lem:identity}.
\end{proof}
\subsection{Proof of Corollary~\ref{cor:zeros-vs-growth}}
In this section we show that Theorem~\ref{thm:discrete-abs-mono} immediately implies  that   harmonic functions of polynomial growth 
are polynomials. At the same time it gives a quantitative estimate on the dimension of the space harmonic polynomials of degree at most~$M$.
This gives a simple proof of a well known result.
\begin{proof}[Proof of Corollary~\ref{cor:zeros-vs-growth}]
By the assumption there exist $C, D > 0$ such that
$$\forall x\in\Zb^d,\quad |u(x)|<C|x|^{M}+D $$
We  estimate~$Q_u$:
\begin{equation}
\label{ineq:Qu-above}
Q_u(n)=\sum_x u(x)^2 p(n, x) 
=\sum_{|x|\leq n}u(x)^2 p(n, x) \leq
(Cn^M+D)^2.
\end{equation}
On the other hand,
by Theorem~\ref{thm:discrete-abs-mono} and Theorem~\ref{thm:finite-computation} we know
that
\begin{equation}
\label{ineq:Q_u-below}
\forall k, n\in\Nb_0,\quad Q_u(n) \geq (\Delta^k u^2)(0) \binom{n}{k}.
\end{equation}
Inequalities~\eqref{ineq:Qu-above} and~\eqref{ineq:Q_u-below} imply
\begin{equation} 
\label{eq:high-laplace-vanish}
\forall k>2M,\quad \Delta^k (u^2)(0) = 0.
\end{equation} 
To see that $u$  is a polynomial, observe that~\eqref{eq:high-laplace-vanish} and Remark~\ref{rem:derivatives} imply that 
\[ \forall k>2M\ \forall s_1,\dots,s_{k} \in S,\quad u_{s_1s_2\dots s_k}(0) = 0\] and hence
\[  \forall s_1,\dots, s_{2M+1} \in S,\quad u_{s_1 s_2 \dots s_{2M+1}} \equiv 0.\]
i.e. $u$ is a polynomial of degree at most $2M$.
But of course, due to the growth assumption on $u$, 
$u$ is a polynomial of degree at most $M$.

Now, assume $u$ vanishes on the ball $B_M$. Then,
\begin{equation}
\label{eq:low-laplace-vanish}
\forall 0\leq k\leq M,\quad \Delta^k (u^2)(0) = 0.
\end{equation}
In addition, by Remark~\ref{rem:derivatives}
\begin{equation}
\label{eq:high-M-laplace-vanish}
\forall k > M,\quad \Delta^k (u^2)(0) = 0.
\end{equation}
We conclude from Theorem~\ref{thm:finite-computation},~\eqref{eq:low-laplace-vanish}
and~\eqref{eq:high-M-laplace-vanish} that
\begin{equation*}
\forall n\in\Nb_0,\quad Q_u(n)=0,
\end{equation*}
which immediately implies $u\equiv 0$.
\end{proof}

\section{Proof of Theorems~\ref{cor:continuous-propagation-of-smallness}-\ref{thm:propagation-with-error-omitted}, and~\ref{thm:discrete-propagation-of-smallness-1-2-5}}
\label{sec:logconv}
In this section we deduce from the absolute monotonicity of
the discrete Agmon function (Theorem~\ref{thm:discrete-abs-mono}) logarithmic convexity type results. 

For convenience we define
\begin{definition}
A function $f:(0,\infty)\to(0,\infty)$ is \emph{log-convex on a logarithmic scale}  (LCLS) if $t\mapsto\log f(e^t)$ is a convex function of $t\in\Rb$. If $f$ is continuous then $f$ is LCLS if $f(2r) \leq \sqrt{f(r) f(4r)}$ for all $r > 0$. 
\end{definition}

We begin with
\begin{proof}[Proof of Theorem~\ref{cor:continuous-propagation-of-smallness}]
By Theorems~\ref{thm:bernstein} and~\ref{thm:finite-computation} 
$$Q_{c, u}(t)= \sum_{k=0}^\infty a_k t^k$$ where all $a_k \geq 0$. 
 Clearly, $a_k t^k$ is LCLS. It is well known that a sum and a limit of log-convex functions is log-convex (see Lemma~\ref{lem:additive} below).
\end{proof}

We now move to the proof of Theorem~\ref{thm:discrete-propagation-of-smallness}.
To explain also the case $\alpha=1/2$ of Theorem~\ref{thm:discrete-propagation-of-smallness-aspect-ratio} we prove a slightly more general version:
\begin{theoremA}
Let $P>1$. Let $u:B_{P^2 R}\to\Rb$ be harmonic and let 
$0\leq\eps\leq 0.5$. Then
\begin{equation}
\label{ineq:propagation-with-error-P}
Q_u(\lfloor Pn \rfloor)\leq \sqrt{e^{n^{-2\eps}}Q_u(n)Q_u(\lceil P^2 n\rceil)}+ P^{-n^{0.5-\eps}}Q_u(\lceil P^2 n \rceil)
\end{equation}
for all $4P^2\leq n\leq R$.
\end{theoremA}
\begin{proof}[Proof of Theorem~\ref{thm:discrete-propagation-of-smallness}']
By Theorems~\ref{thm:high-powers-laplacian}
and~\ref{thm:finite-computation}
 for $n \leq PR$ we know that
\[Q_u(n) = 
\sum_{k=0}^{\infty} a_k \binom{n}{k}  = \sum_{k=0}^{PR} a_k \binom{n}{k}\] 
where $a_k \geq 0$. We observe that inequality~\eqref{ineq:propagation-with-error-P} is additive
(see Lemma~\ref{lem:additive}). Hence, it suffices to show
\begin{equation}
\label{ineq:prop-smallness-simplified}
\binom{\lfloor Pn\rfloor }{k} \leq \sqrt{e^{n^{-2\ep}}\binom{n}{k}\binom{\lceil P^2 n\rceil  }{k}}+P^{-n^{0.5-\ep}}  \binom{\lceil P^2n\rceil }{k}\ 
\end{equation}
for all $k, n\in \Nb_0$.
It can be easily checked that~\eqref{ineq:prop-smallness-simplified}
is satisfied for $k=0, 1$ and all $n$.
From this point on, we assume $k\geq 2$.

Suppose that $n^{1-2\ep} \geq k^2$:
\begin{multline}
\frac{\binom{\lfloor Pn\rfloor}{k}^2}{\binom{n}{k}\binom{\lceil P^2n\rceil }{k}} 
\leq \frac{\prod_{j=0}^{k-1} {(Pn-j)^2}}{\prod_{j=0}^{k-1}(n-j)(P^2n-j)} = \prod_{j=0}^{k-1} \frac{P^2n^2 - 2Pnj + j^2}{P^2n^2 - (P^2+1)nj + j^2}\\
= \prod_{j=0}^{k-1} \left( 1+ \frac{(P-1)^2nj}{P^2n^2-(P^2+1)nj+j^2} \right) \leq  \prod_{j=0}^{k-1} \left(1 + \frac{(P-1)^2j}{P^2 n- (P^2+1)k}\right) \\ 
\stackrel{(*)}{\leq} \prod_{j=0}^{k-1} \left(1 + \frac{j}{n}\right)   \leq \left(1+\frac{k}{n}\right)^k = \left(1+\frac{1}{n/k}\right)^{\frac{n}{k}\frac{k^2}{n}} \leq e^{\frac{k^2}{n}} \leq e^{n^{-2\ep}}\ , 
\label{ineq:propagation-large-n}
\end{multline}
where in $(*)$ we used $P^2 n -(P^2+1)k  > (P-1)^2n$ which holds as long as $n \geq k^2$ and $n \geq 4P^2$.  

On the other hand, if $n^{1-2\ep} \leq k^2 $ then 
\[ \frac{\binom{\lfloor Pn\rfloor }{k}}{\binom{\lceil P^2n\rceil }{k}} \leq \prod_{j=0}^{k-1} \frac{Pn-j}{P^2n-j} \leq P^{-k} \leq P^{-n^{0.5-\ep}}.\]
Thus, we have in fact proved a slightly stronger inequality than~\eqref{ineq:prop-smallness-simplified}, namely
\[\binom{\lfloor Pn\rfloor }{k} \leq \max\left\{ \sqrt{e^{n^{-2\ep}} \binom{n}{k}\binom{\lceil P^2n\rceil }{k}}, P^{-n^{0.5-\ep}}\binom{\lceil P^2n\rceil}{k}\right\}. \]
\end{proof}

\begin{proof}[Proof of Theorem~\ref{thm:propagation-with-error-omitted}]
 By Theorem~\ref{thm:finite-computation} we can write
 $$\forall n\in\Nb_0,\quad Q_u(n)=\sum_{k=0}^M \Delta^k (u^2)(0)\binom{n}{k}.$$
 Indeed,  Remark~\ref{rem:derivatives} shows that if $k>m$ then
  $\Delta^k (u^2)(0)=0$.
  Since $n^{1-2\ep} \geq M^2$ and $n > 16$, for any $0\leq k\leq M$  inequality~\eqref{ineq:propagation-large-n} with $P=2$ applies and yields 
\[\binom{2n}{k} \leq \sqrt{e^{n^{-2\ep}} \binom{n}{k}\binom{4n}{k}}\ . \] 
Theorem~\ref{thm:high-powers-laplacian} and Lemma~\ref{lem:additive} complete the proof.
\end{proof}

\begin{proof}[Proof of Theorem~\ref{thm:discrete-propagation-of-smallness-1-2-5}]
While the idea is similar to the idea of  the proof of Theorem~\ref{thm:discrete-propagation-of-smallness}
the estimates here are simpler.

Let $n\geq (1+1/(4\delta))k$.
Then
$$\frac{\binom{2n}{k}^2}{\binom{n}{k}
\binom{\lceil4(1+\delta)n\rceil}{k}}
\leq \prod_{j=0}^{k-1} 
\frac{(2n-j)^2}{(n-j)(4(1+\delta)n-j)}\leq 1,$$
where the last estimate is true since $n\geq (1+1/(4\delta))k$.

On the other hand, if $n\leq (1+1/(4\delta))k$
then
$$\frac{\binom{2n}{k}}{\binom{\lceil4(1+\delta)n\rceil}{k}}
=\prod_{j=0}^{k-1} \frac{2n-j}{4(1+\delta)n-j}\leq 2^{-k}\leq 2^{-4\delta n/(4\delta+1)}\ .$$
To summarize
$$\binom{2n}{k} \leq \max\left\{\sqrt{\binom{n}{k}\binom{\lceil 4(1+\delta)n\rceil }{k}}, 2^{-4\delta n/(4\delta+1)}\binom{\lceil 4(1+\delta)n\rceil}{k}\right\}. $$
The claim now follows as in the proof of Theorem~\ref{thm:discrete-propagation-of-smallness}'.
\end{proof}
It remains to prove
\begin{lemma}\label{lem:additive}
 Let $f_1, f_2 : \Nb_0\to\Rb^+$ be functions such that
$$ f_i( n_2 ) \leq C_1 \sqrt{f_i(n_1) f_i(n_3)} 
+ C_2 E(n_1)f_i( n_3 ) $$
for some $n_1, n_2, n_3\in \Nb_0$, some constants $C_1,C_2 > 0$ and some function $E: \Nb_0\to\Rb^+$.  Then the same holds for $f = f_1+f_2$.
\end{lemma}
\begin{proof}
The Lemma follows if we can show
\[ \sqrt{ab} + \sqrt{cd} \leq \sqrt{a+c} \cdot \sqrt{b+d}\]
for all $a, b, c, d\geq 0$. This is true due to the Cauchy-Schwarz inequality.
\end{proof}

\section{Proof of Theorem~\ref{thm:weak-conjecture}}
\label{sec:weak-conj}
Let us now give examples of harmonic functions
which exhibit the optimality of the error term
in Theorem~\ref{thm:discrete-propagation-of-smallness} if we do not
fix the dimension of the lattice.
The idea is to construct a function $u$ 
for which $Q_u(n)=\binom{n}{k}$ for any given~$k$,
and to analyze the respective error in a three circles theorem.

\begin{proof}[Proof of Theorem~\ref{thm:weak-conjecture}]
Let $k\leq d$. Define $u_k:\Zb^d\to\Rb$ by 
$u_k(x_1,\dots,x_d) = x_1 x_2 \cdots x_k$. It is easy to check that $u_k$ is harmonic
when we take the generating set, $S$, of $\Zb^d$ to be
the standard one, i.e., we consider $\Zb^d$ as the free abelian
group generated by $S_0=\{e_1, \ldots, e_d\}$ and we
take $S=S_0\cup{-S_0}$.
By Theorem~\ref{thm:finite-computation} and Remark~\ref{rem:derivatives} we see that 
\[ Q_{u_k}(n) = \sum_{l = 0}^k   \frac{1}{|S|^l} \sum_{s_1,s_2,\dots,s_l \in S} (u_{s_1s_2\dots s_l})^2(0)\binom{n}{l}.\]
However $u_{s_1 s_2 \dots s_l}(0) = 0$ unless $l=k$ and 
the $s_i$'s generate $\oplus_{j=1}^k \Zb e_j$.
In that case $u_{s_1 s_2\dots s_k}(0)^2 = 1$. Thus 
\[ Q_{u_k}(n) = c_{d, k} \binom{n}{k}, \]
where $c_{d, k}=k!/d^k$.
The theorem now follows from Proposition~\ref{prop:binomial-error} below.
\end{proof}
\begin{remark}
\label{rem:possibleQ}
We can in fact see from the previous proof that if $P$ is any  polynomial of degree at most~$d$ that is absolutely monotonic in the discrete sense, then there exists a harmonic polynomial $u: \Zb^d \to \Rb$ such that $Q_u= P$.
\end{remark}
\begin{proposition}
\label{prop:binomial-error}
Let $C>0, \eps>0$. Let $k>k_0(C, \eps)$,
and $n=k^2/\log k$.
Then,
$$\binom{2n}{k}>C\sqrt{\binom{n}{k}\binom{4n}{k}} + 2^{-n^{0.5+\eps}}\binom{4n}{k}\ .$$
\end{proposition}
\begin{proof}[Proof of Proposition~\ref{prop:binomial-error}]
\begin{align*} 
\frac{\binom{2n}{k}}{\binom{4n}{k}} &= \prod_{j=0}^{k-1} \frac{2n -j}{4n-j}= \frac{1}{2^k} \prod_{j=0}^{k-1}\left( 1 - \frac{j}{4n-j}\right)\\  &> \frac{1}{2^k}\left(1-\frac{k}{4n-k}\right)^k
\sim 2^{-k}k^{-1/4} \gg 2^{-n^{0.5+\eps}},
\end{align*}
where the last estimate is true since  $k\ll n^{0.5+\eps/2}$.

On the other hand since $k \gg n^{0.5}$ we have,
\begin{multline*} \frac{\binom{2n}{k}^2}{\binom{n}{k}\binom{4n}{k}} = \prod_{j=0}^{k-1} \frac{(2n- j)^2}{(n - j)(4n -j)}  
=  \prod_{j=0}^{k-1} \left(1 + \frac{jn}{(n-j)(4n-j)} \right) 
\geq \\ \geq  \prod_{j=0}^{k-1} \left(1 + \frac{j}{4n} \right)
\geq\left(1+\frac{k}{8n}\right)^{k/2}
 \sim e^{k^2 / (16n)} =  k^{1/16} \gg 1. \end{multline*}
\end{proof}

\section{Conjecture~\ref{conj:sharp-error}: A discussion}
\label{sec:conj}
We explain our intuition and motivation in Conjecture~\ref{conj:sharp-error}. First observe that
the proof of Theorem~\ref{thm:weak-conjecture}
shows that if we could construct for a fixed dimension $d$ and any $k\in\Nb$ a 
harmonic function $u$ such that 
$Q_u(n)=\binom{n}{k}$ then
the conjecture would follow by Proposition~\ref{prop:binomial-error}.
However it seems that this is only possible to do for $k\leq d$ (see Remark~\ref{rem:possibleQ}).
A natural way to try to construct a polynomial $u$ for which $Q_u$ would approximate $\binom{n}{k}$ is
to start with the polynomial $u_k=\Re z^k$ for which
$q_{u_k}(r)=C r^{2k}$ and to discretize it. 
Here we recall an algorithm for the discretization process 
due to Jerison-Levine-Sheffield~\cite{jls-duke14},
the origins of which can be found in~\cite{lovasz-discrete-analytic-survey}.
\subsection{The Jerison-Levine-Sheffield construction}
\label{sec:correspondence}
We essentially describe the construction from~\cite{jls-duke14}.
\vspace{1ex}

\begin{notation*}
Given a sequence of functions $(F_k)_{k=0}^{\infty}:\Zb\to\Rb$, and a multi-index $\alpha\in\Nb_0^{d}$, we define
the function $F_{\alpha}:\Zb^d\to\Rb$ by
$$\forall x=(x_1,\ldots, x_d)\quad F_\alpha(x):=\prod_{l=1}^d F_{\alpha_l}(x_l).$$
\end{notation*}

\begin{theorem}[Correspondence principle]
\label{thm:correspondence}
Let $F_k:\Zb\to\Rb$, $k\in\Nb_0$, be a sequence of functions
such that
$\Delta F_1=\Delta F_0 = 0$, and $\Delta F_k = AF_{k-2}$ for all $k\geq 2$ and for some $A\in\Rb$.
Let
$$P(x)=\sum_{|\alpha|\leq M} a_{\alpha} \frac{x^{\alpha}}{\alpha!}\ $$
be a harmonic polynomial in $\Rb^d$.
Then, 
the polynomial
$$P^{\Zb}(x)=\sum_{|\alpha|\leq M} a_{\alpha} F_{\alpha}(x),$$
is harmonic in $\Zb^d$.
\end{theorem}

The proposition below gives a concrete sequence of functions $F_k:\Zb^d\to\Rb$
which can be plugged into the correspondence principle.
\begin{proposition}
\label{prop:Fk}
Let $F_k:\Zb\to\Rb$, $k\in\Nb_0$, be defined as follows:
\begin{align*}
F_0(x)&=1, \\ 
\forall k>0\ F_k(x) &= \binom{ x+ \frac{k-1}{2}}{k} = \frac{1}{k!}
\prod_{j=0}^{k-1} \left(x+\frac{k-1}{2}-j\right).
\end{align*}

Then
\begin{enumerate}
\item \label{prop:1st-derivative} $F_{k}'(x)=F_{k-1}(x+\frac{1}{2})$
\item \label{prop:2nd-derivative}
 $\Delta F_1=\Delta F_0 = 0$ and $\Delta F_k = \frac{1}{2}F_{k-2}$ for all
$k\geq 2$.
\end{enumerate}
\end{proposition}
\begin{proof}
First we check that part~\ref{prop:2nd-derivative}
follows from part~\ref{prop:1st-derivative}.
In fact,  by part~\ref{prop:1st-derivative} $F_k''(x)=F_{k-1}'(x+\frac{1}{2})=F_{k-2}(x+1)$.
It remains to observe that $\Delta F(x)=\frac{1}{2}F''(x-1)$.
To prove part~\ref{prop:1st-derivative} we observe
that
\begin{align*}
F_k(x+1) &=\frac{1}{k}\left(x+\frac{k+1}{2}\right)F_{k-1}\left(x+\frac{1}{2}\right) \\
\mbox{and}\quad F_k(x) &= \frac{1}{k}\left(x-\frac{k-1}{2}\right)F_{k-1}
\left(x+\frac{1}{2}\right)\ .
\end{align*}
\end{proof}
\begin{remark*}
The family $F_k(x)=\binom{x+\lfloor k/2 \rfloor}{k}$
is implicitly used in~\cite{lovasz-discrete-analytic-survey}.
\end{remark*}
\subsection{Specialization of Conjecture~\ref{conj:sharp-error}}
Let $F_k:\Zb\to \Rb$ be as in Proposition~\ref{prop:Fk}. Define
\begin{align*}
 \forall k\geq 0\quad\quad S_k(x,y)&:= \sum_{j=0}^{\lfloor k/2 \rfloor} (-1)^j F_{k-2j}(x) F_{2j}(y),\\
\mbox{and }\ \forall k\geq 1\quad\quad T_k(x,y)&:= \sum_{j=0}^{\lfloor (k-1)/2 \rfloor} (-1)^j F_{k-(2j+1)}(x)F_{2j+1}(y).
\end{align*}
Then, $S_k, T_k$ are harmonic in $\Zb^2$. This can be immediately seen
 from the fact
that $\Re (x+iy)^k$ and $\Im (x+iy)^k$ are harmonic in $\Rb^2$ and
Theorem~\ref{thm:correspondence}.

We believe that $S_k$ is a family of harmonic functions which gives the optimal error term in Theorem~\ref{thm:discrete-propagation-of-smallness}, namely,
\begin{conjecture} Let $C>0$. Then
$$Q_{S_k}(2n)>C\sqrt{Q_{S_k}(n)Q_{S_k}(4n)}+2^{-n^{0.5+\eps}}Q_{S_k}(4n)$$
for $k$ large enough and $n\sim k^2/\log k$.
\end{conjecture}

\vfill
\newpage
\noindent Gabor Lippner,\\
Department of Mathematics, Harvard University,\\
 Cambridge, MA, USA\\
\smallskip
\texttt{lippner@math.harvard.edu}
\vspace{1ex}

\noindent Dan Mangoubi,\\ 
Einstein Institute of Mathematics,
Hebrew University, Givat Ram,\\
Jerusalem,
Israel\\
\smallskip
\texttt{mangoubi@math.huji.ac.il}
\end{document}